\documentclass[12pt, reqno]{amsart}
\usepackage{amsmath, amsthm, amscd, amsfonts, amssymb, graphicx, color, dsfont, mathrsfs, enumerate, mathtools, xcolor}
\usepackage{inputenc}
\usepackage[all]{xy}
\usepackage[bookmarksnumbered, colorlinks, plainpages]{hyperref}
\hypersetup{colorlinks=true,linkcolor=red, anchorcolor=green, citecolor=cyan, urlcolor=red, filecolor=magenta, pdftoolbar=true}
\makeatother
\normalsize
\newtheorem{theorem}{Theorem}[section]
\newtheorem{lemma}[theorem]{Lemma}
\newtheorem{proposition}[theorem]{Proposition}
\newtheorem{corollary}[theorem]{Corollary}
\theoremstyle{definition}

\theoremstyle{remark}
\newtheorem{remark}[theorem]{Remark}
\numberwithin{equation}{section}

\allowdisplaybreaks

\newcommand{\N}{\mathcal N}
\newcommand{\B}{\mathcal B}
\newcommand{\F}{\mathcal F}
\newcommand{\A}{\mathcal A}

\numberwithin{equation}{section}

\begin{document}
\title [The principle of local reflexivity and an extension  of ...]{The principle of local reflexivity and an extension  of  the identity $ \mathcal B(E,X^{**})\cong \B(E,X)^{**}$}
\author[R. Faal]{Ramin  Faal}
\author[H.R. Ebrahimi Vishki]{Hamid Reza  Ebrahimi Vishki}
\address{ Department of Pure Mathematics, Ferdowsi University of Mashhad, P.O. Box 1159, Mashhad 91775, IRAN.}
\email{faal.ramin@yahoo.com}
\address{ Department of Pure Mathematics and Center of Excellence in Analysis on Algebraic Structures (CEAAS), Ferdowsi University of Mashhad, P.O. Box 1159, Mashhad 91775, IRAN.}
\email{vishki@um.ac.ir}
\begin{abstract} By using  the Principle of Local Reflexivity (PLR), we prove  that for every two Banach spaces $E$ and $X$ there exists a suitable  ultrafilter $\mathcal{U}$ such  that $ \mathcal{F}(E,X)^*,$ the dual space  of the finite rank operators, can be isomorphically identified with certain quotient of the ultrapower space $(E\widehat{\otimes} X^*)_\mathcal{U}$, of the projective tensor product space $E\widehat{\otimes} X^*.$  This generalizes  the   identity $\mathcal B(E,X^{**})\cong\mathcal B(E,X)^{**}$, where  $E$ is finite-dimensional. We then serve our main result  to improve some    results on the reflexivity of $\B(E,X)$, the space of all bounded linear operators,  by showing that: if $\B(E,X)$ is reflexive, then $\B(E,X)=\A(E,X)$, the  space of all approximable  operators. This particularly implies that,  $\B(E)$ is reflexive if and only if $E$ is finite-dimensional.   Finally, as   more  by-products of the PLR, some    generalizations of the classical   Goldstine weak$^*$-density theorem are also included.
\end{abstract}
\maketitle
\section{Introduction}
The classical Dean identity, that was first found in the theory of tensor product of Banach spaces in \cite{Gr,SC}, reads as follows:
\begin{theorem}\label{Dean Identity}{\rm (Dean identity, \cite{D}).}
Let $E$ be a finite-dimensional normed space and let $X$ be a Banach space. Then $ \mathcal B(E,X^{**})\cong\mathcal B(E,X)^{**}$, where $\mathcal B(E,X)$ denotes the space of all bounded linear operators from $E$ to $X$.
\end{theorem}
This identity has been particularly fruitful because it can be easily adapted to
different categorical contexts by replacing the bounded linear operators with some other family of maps. 
To the best of our knowledge, the most technical proof of Theorem \ref{Dean Identity}, based on the nontensor product techniques, was given by Dean \cite{D}, and this is why, it was termed Dean identity in \cite{CG}.
Dean  \cite{D} also served  this  identity to provide a  proof for  the well-known principle of local reflexivity,  stating that the second dual $X^{**}$ of a Banach space $X$ ``locally" behaves the same as $X$ itself, that  initially established by Lindenstrauss and Rosenthal \cite{LR}.
\begin{theorem}{\rm (Principle of Local Reflexivity (PLR), \cite{LR}).}\label{PLR}
Let $X$ be a Banach space and let $E$ and $ F$ be finite-dimensional subspaces of $X^{**}$ and $X^*,$ respectively. Then for each $\varepsilon>0$ there exists a one-to-one linear map $S:E\to X$ such that $\|S\|, \|S^{-1}\|<1+\varepsilon,$ $S|_{E\cap X}=id_{E\cap X}$ and $f(S(e))=e(f)$ for all $e\in E$ and $f\in F.$
\end{theorem}
The PLR is a fundamental (and very useful) result in the theory of Banach
spaces. Many improvements and generalizations on this principle have been given in the literature; see, for example, \cite{B,CD,JL,MA,O,OP} and references therein.\smallskip

 A constructive proof for the PLR was given by Dean,  \cite[Theorem 1]{D}, in which he nicely served  his identity (Theorem \ref{Dean Identity}). Here, however, we focus on the reverse direction and derive a generalized version of Dean identity, whose proof is basically relied on the PLR.

In Section 2, we  prove Theorem \ref{ultrafilter}, which  is an extension of  Theorem \ref{Dean Identity} (see Corollary \ref{SD}). In Section 3 we focus on the reflexivity of $\B(E,X)$. We actually serve  Theorem \ref{ultrafilter}  to present Theorem \ref{finite} improving some  results of Holub \cite{H} and Kalton \cite{K}. In particular, we derive Corollary \ref{finite1} characterizing the reflexivity of $\B(E)$. We then  serve  Theorem \ref{ultrafilter} to give  Theorem \ref{CSD} and Corollary \ref{CSD1}, which can be viewed  as some converses of Corollary \ref{SD}.  Section 4 is devoted to some by-products of the PLR, presenting some extensions of the Goldstine weak$^*$-density theorem.
\section{An extension  of Dean identity}
In this section, we extend Theorem \ref{Dean Identity} by removing the restriction $``\dim(E)<\infty"$. Roughly speaking, we show that for every two Banach spaces $E$ and $X$, an ultrapower $(E\widehat{\otimes}X^*)_{\mathcal U}$ of $E\widehat{\otimes}X^*$ can be identified with $\mathcal F(E,X)^*,$ where $\F(E,X)$ stands for the finite rank operators in $\mathcal B(E,X).$\smallskip

Let us first provide some prerequisites on the ultrapowers of a Banach space. Let $X$ be a Banach space, let $I$ be an indexing set, and let $\mathcal{U}$ be an ultrafilter on $I.$ We define the ultrapower $X_\mathcal{U}$ of $X$ with respect to $\mathcal{U},$ by the quotient space
$X_\mathcal{U}=\ell^\infty(X,I)/\N_\mathcal{U},$
where $\ell^\infty(X,I)$ is the Banach space
\[\ell^\infty(X,I)=\{(x_{\alpha})_{\alpha\in I}\subseteq X: \|(x_{\alpha})\|=\sup_{\alpha\in I}\|x_{\alpha}\|<\infty\}\]
and $\N_\mathcal{U}$ is the closed subspace $\N_\mathcal{U}=\{(x_{\alpha})_{\alpha\in I}\in\ell^\infty(X,I): \lim_{\mathcal{U}}\|x_{\alpha}\|=0\}.$
Then the quotient norm coincides with the norm $\|(x_{\alpha})\|_\mathcal{U}:=\lim_\mathcal{U}\|x_{\alpha}\|$,
and $X$ can be identified with a closed subspace of $X_\mathcal{U}$ via the canonical isometric embedding $x\mapsto (x)_\mathcal{U}.$ There is also a norm-decreasing map $\sigma : X_\mathcal{U}\rightarrow X^{**}$ such that
$\sigma ( (x_\alpha)_\mathcal{U}) = w^*-\lim_\mathcal{U} \kappa_X (x_\alpha), (x_\alpha)_\mathcal{U}\in X_\mathcal{U},$
where $\kappa_X $ is the canonical embedding of $X$ into $X^{**}$. Ample information about ultrapowers, with more details, can be found in \cite{He}. We quote the following known result, which will be needed in the proof of Theorem \ref{ultrafilter}.


\begin{proposition} [{Heinrich, \cite[Proposition 6.7]{He}}]\label{add}
Let $X$ be a Banach space. Then there exist an ultrafilter $\mathcal{U}$ and a linear isometric embedding $K : X^{**}\rightarrow X_\mathcal{U}$ such that $\sigma \circ K$ is the identity on $X^{**}$ and $K\circ\kappa_X$ is the canonical embedding of $X$ into $X_\mathcal{U}.$ Thus $K\circ\sigma$ is a norm-one projection of $X_\mathcal{U}$ onto $K(X^{**})$.
\end{proposition}
As it was described in \cite{FE}, it is worth to mention that the ultrafilter $\mathcal U$, used in the above proposition, is countably incomplete.\medskip

We are now ready to prove the main result of this section, which can be viewed as an infinite-dimensional version of Theorem \ref{Dean Identity} (see Corollary \ref{SD}). Hereafter $\widehat{\otimes}$ stands for the projective tensor product, and $\cong$ shows the isometric isomorphic between Banach spaces. 
\
\begin{theorem}\label{ultrafilter}
Let $E$ and $X$ be Banach spaces. Then there exists an ultrafilter $\mathcal{U}$ such that 
 \[\F(E,X)^*\cong (E\widehat{\otimes} X^*)_\mathcal{U}/W_0,\]
  where $W_0=\left\{ (u_\alpha)_\mathcal{U}\in (E\widehat{\otimes} X^*)_\mathcal{U} : \lim_\mathcal{U}u_i(T)=0,T\in \mathcal{F}(E,X)\right\}$.
\end{theorem}
\begin{proof} First note that if $T\in \mathcal F(E,X)$ with $\dim T(E)=n,$ then there exists a subset $\lbrace e_1,\ldots , e_n\rbrace$ of $E$ such that $\ \langle e_1,\ldots ,e_n\rangle \oplus Ker(T) =E$ (here $\langle e_1,\ldots ,e_n\rangle$ stands for the linear span of $\lbrace e_1,\ldots , e_n\rbrace$). Based on this fact, we define an order $\leq$ on $\mathcal F(E,X)$ as follows:

For $S,T\in \mathcal F(E,X),$ we say $T\leq S$ if there exists a subset $ \lbrace e_1,\ldots, e_n,\ldots ,e_m\rbrace$ of $E$ such that $\langle e_1,\ldots ,e_n\rangle \oplus Ker(T) =E$ and
$\langle e_1,\ldots ,e_n,\ldots ,e_m\rangle \oplus Ker(S)=E$, with $T(e_i)=0\ {\rm for} \ i>n.$

Then obviously we have $T\leq T$. Suppose that $T\leq S$ and $S\leq U$, then there exist $\lbrace e_1,\ldots, e_n,\ldots ,e_m\rbrace\subseteq E$ and $ \lbrace p_1,\ldots, p_m,\ldots ,p_l\rbrace\subseteq E$ such that
\begin{align*}
&\langle e_1,\ldots ,e_n\rangle \oplus Ker(T) =E,
\langle e_1,\ldots ,e_n,\ldots ,e_m\rangle \oplus Ker(S)=E\ {\rm with} \ T(e_i)=0\ {\rm for} \ i>n \ {\rm and} \\
&\langle p_1,\ldots ,p_m\rangle \oplus Ker(S) =E,
\langle p_1,\ldots ,p_m,\ldots ,p_l\rangle \oplus Ker(U)=E\ {\rm with} \ S(p_i)=0\ {\rm for} \ i>n.
\end{align*}
We may, without loss of generality, assume that $\langle p_1,\ldots ,p_n\rangle =\langle e_1,\ldots ,e_n\rangle$. Then we set $p_{n+1}:=\sum_{i=1}^n \alpha_i e_i + q_{n+1}$ and $p_{n+2}:=\sum_{i=1}^n \beta_i e_i + \beta_{n+1}q_{n+1} +q_{n+2}$, where $q_{n+1}, q_{n+2}\in Ker (T)$ with $\|q_{n+2}\|=1$, $q_{n+2}\notin\langle q_{n+1}\rangle$, and $\beta_{n+1}$, $\alpha_i, \beta_i$ are scalars for each $i=1,\ldots,n.$
Continuing this procedure, we arrive at the linearly independent set $\lbrace e_1, \ldots ,e_n,q_{n+1},\ldots ,q_l\rbrace\subseteq E$ satisfying
\[\langle e_1, \ldots ,e_n,q_{n+1},\ldots ,q_l\rangle\oplus Ker(U)=E\ \mbox{with}\ T(q_i)=0\ \mbox{for}\ i>n,\] and this simply implies that $T\leq U.$

We also need to show that, for every $S,T\in \mathcal \mathcal \mathcal \mathcal \mathcal F(E,X)$, there exists $U\in\mathcal F(E,X)$ such that $S\leq U$ and $T\leq U$. To do this, let $\lbrace e_1,\ldots , e_n\rbrace\subseteq E$ and $ \lbrace p_1,\ldots , p_m\rbrace\subseteq E$ be such that $\langle e_1,\ldots ,e_n\rangle \oplus Ker(T) =E$ and $\langle p_1,\ldots ,p_m\rangle \oplus Ker(S) =E $, and set $\langle e_1,\ldots ,e_n\rangle \cap \langle p_1,\ldots ,p_m\rangle =\langle r_1,\ldots ,r_k\rangle$. Similar to the above procedure, for each $1\leq u\leq n$ and $1\leq v\leq m,$ we can replace $\lbrace e_1,\ldots, e_u,r_1,\ldots ,r_k,p_1,\ldots,p_v\rbrace$ by the set $\lbrace q_1,\ldots, q_{\ell},r_1,\ldots ,r_k,s_1,\ldots,s_t\rbrace$, where $q_i\in Ker(S), 1\leq i\leq\ell,$ and $s_j\in Ker(T), 1\leq j\leq t$. Now there exist a subset $\{ x_1,\ldots , x_{\ell+k+t}\}$ of $X$ and an isomorphism $U$ from $\langle q_1,\ldots, q_\ell,r_1,\ldots ,r_k,s_1,\ldots,s_t\rangle$ to $\langle x_1,\ldots , x_{\ell+k+t}\rangle$ that can be naturally extended to $E,$ and this confirms that $T\leq U$ and $S\leq U$.\medskip

Now let $\mathcal U$ be the ultrafilter inducing by refining the order filter on the set $I=\lbrace T\in\mathcal F(E,X) : \|T\|=1\rbrace$ and set $I_T:=\{S: S\geq T\}$. Then we define $\Psi : (E\widehat{\otimes} X^*)_\mathcal{U}\to \mathcal F(E,X)^*$ by the rule
\begin{equation}\label{psi}
\Psi \Big(\big(\sum_{i=1}^\infty a_{\alpha ,i}\otimes x^*_{\alpha ,i}\big)_\mathcal{U}\Big)(T)=\lim_{\mathcal U}\sum_{i=1}^\infty x_{\alpha ,i}^*(T(a_{\alpha ,i})),
\end{equation}
where $\left(\sum_{i=1}^\infty a_{\alpha ,i}\otimes x^*_{\alpha ,i}\right)_\mathcal{U}\in (E\widehat{\otimes} X^*)_\mathcal{U}$ and $T\in\mathcal F(E,X).$
We first show that $\Psi$ is surjective. Let $T\in I$ and let $\lbrace e^T_1,\ldots , e^T_{n_T}\rbrace\subseteq E$ be such that $\langle e^T_1,\ldots , e^T_{n_T}\rangle\oplus Ker(T)=E$. Then for each $x\in X$ and $1\leq i\leq n_T$, there exists $T_{e^T_i,x}\in \mathcal F(E,X)$ such that $T_{e^T_i,x}(e^T_i)=x$ and $T_{e^T_i,x}(y)=0$ when $y\notin\langle e^T_i\rangle$. It follows that $T=\sum_{i=1}^{n_T}T_{e^T_i,T(e^T_i)}.$ We associate a functional $\phi_T\in \mathcal F(E,X)^*$ to each $\phi\in \mathcal F(E,X)^*$ by the role $\phi_T(U) =\sum_{i=1}^{n_T}\phi(T_{e^T_i,U(e^T_i)})$ for each $U\in \mathcal F(E,X)$. It is easy to check that $\phi_T$ is independent of the choice of $e^T_1,\ldots , e^T_{n_T}$.

Then we claim that $\phi =(\phi_T)_\mathcal{U}$. Indeed, if $S\geq T$, then there exists $\{e^S_1,\ldots ,e^S_{n_T},\ldots ,e^S_{m_S}\}\subseteq E$ such that $\langle e^S_1,\ldots ,e^S_{n_T},\ldots ,e^S_{m_S}\rangle \oplus Ker(S)=E,$ with $e^S_i=e^T_i$ for $i=1,\ldots ,n_T$ and $T(e^S_i)=0,$ when $i>n_T.$
From this, we get
\[
\phi_S(T)= \sum_{i=1}^{m_S} \phi \big(S_{e^S_i,T(e^S_i)}\big)=\sum_{i=1}^{n_T}\phi\big(T_{e^T_i,T(e^T_i)}\big)=\phi_T(T)
\]
and
\[
\phi (T)=\phi \Big(\sum_{i=1}^{n_T}T_{e^T_i,T(e^T_i)}\Big)=\phi_T(T).
\]
The latter identity follows that $\phi=(\phi_T)_\mathcal{U}.$ We also have
\begin{align}\label{onto}
\phi (T) =\phi _T(T) = \sum_{i=1}^{n_T}\phi \big(T_{e^T_i,T(e^T_i)}\big)=\sum_{i=1}^{n_T} f_{T,i}\big(T(e^T_i)\big) =\Psi\Big(\big(\sum_{i=1}^{n_T} e^T_i\otimes f_{T,i}\big)_{\mathcal U}\Big)(T),
\end{align}
where $f_{T,i}\in X^*$ is defined by $f_{T,i}(x)= \phi (T_{e^T_i,x}) (x\in X).$ From \eqref{onto}, we get that $\Psi$ is onto.

If we set $W_0=\ker \Psi,$
then the quotient mapping $\Psi_0: (E\widehat{\otimes} X^*)_\mathcal{U}/W_0\to  F(E,X)^*$ which maps $(u_\alpha)_\mathcal{U}+W_0\in (E\widehat{\otimes} X^*)_\mathcal{U}/W_0 $ to $\Psi((u_\alpha)_\mathcal{U})$ is  bijective. Further,
  for every $u=(u_\alpha)_\mathcal{U}$ in $(E\widehat{\otimes} X^*)_\mathcal{U}$ and $(w_\alpha)_\mathcal{U} \in W_0$, we have 
%
\begin{align*}
\|\Psi_0(u+W_0)\| &=\sup_{T\in {\rm Ball}(\mathcal F(E,X))} |\Psi ((u_\alpha)) (T)| \\
&=\sup_{T\in {\rm Ball}(\mathcal F(E,X))} | \lim_\mathcal{U}u_\alpha (T)| \\
&= \sup_{T\in {\rm Ball}(\mathcal F(E,X))} | \lim_\mathcal{U}u_\alpha + w_\alpha(T)|\\
&\leq \lim_\mathcal{U}\sum_{i=1}^\infty \| u_\alpha + w_\alpha\| \\
&\leq \| u+w \|_\mathcal{U}.
\end{align*}
Since $w\in W_0$ is arbitrary, we have $\|\Psi_0(u+W_0)\|\leq \|u+W_0\|$. Open mapping theorem now implies that $\Psi_0$ is an isomorphism and this completes the proof.
\end{proof}
Since  $\mathcal{A}(E,X)$, the space of all approximable operators from $E$ to $X$,  is the norm closure of $\mathcal{F}(E,X)$, the following corollary comes immediately from Theorem \ref{ultrafilter}.
\begin{corollary}\label{aultrafilter}
Let $E$ and $X$ be two Banach spaces. Then there exists an ultrafilter $\mathcal{U}$ such that \[(E\widehat{\otimes} X^*)_\mathcal{U}/W_0\cong\mathcal{A}(E,X)^*,\] where  $W_0=\left\{ (u_\alpha)_\mathcal{U}\in (E\widehat{\otimes} X^*)_\mathcal{U} : \lim_\mathcal{U}u_i(T)=0,T\in\mathcal{A}(E,X)\right\}.$
\end{corollary}
\begin{remark}\label{RS}
We describe some properties of the ultrafilter $\mathcal U$ satisfying  Theorem \ref{ultrafilter} (and Corollary \ref{aultrafilter}). Recall that $\mathcal U$ is the ultrafilter inducing by refining the order filter on the set $I=\lbrace T\in\mathcal F(E,X) : \|T\|=1\rbrace$ and that $I_T:=\{S: S\geq T\}.$
\begin{enumerate}[\hspace{1em}\rm (a)]
\item In the case where $\dim (E)<\infty,$ we claim that the involved ultrafilter $\mathcal U$ must be principal. To show this, we first note that if $X$ is of finite dimension, then $E\widehat{\otimes} X^*$ is also of finite dimension and so $\mathcal U$ is principal. For the case that $X$ has infinite dimension, we can simply find an injective element $S_0\in\F(E,X)$ such that $T\leq S_0$ (i.e., $S_0\in I_T$) for all $T\in\F(E,X)$ with $\|T\|=1$, consequently $\mathcal U$ is again principal. The principality of $\mathcal U$ provides some simplification in the identity of Theorem \ref{ultrafilter} that will be applied in the proof of Corollary \ref{SD}.

\item We consider the case when $\mathcal B(E,X)=\mathcal F(E,X)$ and $E$ is infinite-dimensional. In this special case, the involved $\mathcal{U}$ is free and countably incomplete. To prove this claim, let $\lbrace e_1,e_2,e_3,\ldots \rbrace$ be a basis for $E$; then there exists an increasing sequence $(T_n)$ of finite rank operators with $\|T_i\|=1$ such that $\langle e_1,\ldots, e_i\rangle\oplus Ker (T_i)=E$. 
If $S\in \bigcap_{i=1}^\infty I_{T_i}$, then there exists a finite-dimensional subspace $F$ with $\dim F=k$ such that $F\oplus Ker (S)=E$. But the fact that $S\in I_{T_{k+1}}$ follows that $\dim F>k,$ which is a contradiction. We thus prove that $ \bigcap_{i=1}^\infty {I}_{T_i}=\emptyset,$ and so $\mathcal{U}$ is free and countably incomplete. \medskip

\end{enumerate}
\end{remark}
Applying Theorem \ref{ultrafilter} for the case where $E$ is finite-dimensional, we come to the following corollary, whose first part  is somewhat stronger than the classical Dean identity (Theorem \ref{Dean Identity}).
\begin{corollary}\label{SD}
Let $E$ be a finite-dimensional normed space and let $X$ be a Banach space. Then \[ E\widehat{\otimes}X^*\cong\mathcal B(E,X)^*.\]
In particular, it follows  Dean identity $\mathcal B(E,X^{**})\cong\mathcal B(E,X)^{**}.$
\end{corollary}
\begin{proof} Since $\dim(E)<\infty$, as we have discussed in Remark \ref{RS}(a), the ultrafilter $\mathcal U$ appeared in Theorem \ref{ultrafilter} should be principal. This follows that  $W_0$  (as appeared in Theorem \ref{ultrafilter}) lies in $E\widehat{\otimes}X^*$ and  we shall see that it must be trivial. To do this, 
Let $u=\sum_{i=1}^n a_i\widehat{\otimes}x_i^*\in W_0$ and $T\in (E\widehat{\otimes}X^*)^*\cong\mathcal{B}(E,X^{**})=\mathcal{F}(E,X^{**})$ be such that $\|u\|\leq |u(T)|+\varepsilon$. By the PLR, Theorem \ref{PLR},  there exists a bounded linear operator  $S:T(E)\to X$ such that $\|S\|,\|S^{-1}\|<1+\varepsilon$, $S|_{T(E)\cap X}= id_{T(E)\cap X}$ and $f(S(T(e)))=f(T(e))$ for all $e\in E$ and $f\in \{x_1^*,\cdots ,x_n^*\}\subseteq X^*$. Since $S\circ T\in \mathcal{F}(E,X)$ and $u\in W_0$, we have
\[
u(T)=\sum_{i=1}^n x_i^*(T(a_i))= \sum_{i=1}^n x_i^*(S(T(a_i)))=u(S\circ T)=0,
\]
 and this confirms  that  $W_0$ is trivial. Consequently, we arrive at the identity $ E\widehat{\otimes}X^*\cong\mathcal B(E,X)^*$. It particularly follows that 
$\mathcal B(E,X^{**})\cong(E\widehat{\otimes} X^*)^*\cong\mathcal B(E,X)^{**},$
as claimed.
\end{proof}\medskip
\section{Reflexivity of $\mathcal{B}(E,X)$}
Here we serve  Theorem \ref{ultrafilter} for investigating  on the reflexivity of  $\mathcal{B}(E,X)$, a problem whose history  goes back to the works of Grothendieck \cite{Gr1} (see also \cite{H}). It was  Ruckle \cite{R} that used a deep theorem of Grothendieck to show that,   if $E$ and $X$ are reflexive  and both $E$ and $X$  have the  approximation property, then $\mathcal{B}(E,X)$ is reflexive if and only if  $\mathcal{B}(E,X)=\mathcal{K}(E,X),$ the space of compact operators. Later Holub \cite{H}  improved Ruckle's result by showing that the result remains valid when either $E$ or $X$ has the approximation property. However, an inspection of Holub's proof of  \cite[Theorem 2]{H} shows that the assumption that $E$ or $X$  have the approximation property was not  used  in the proof of sufficiency. Independently, Kalton \cite{K} used the uniform boundedness principle to  give a direct proof for the  sufficiency, without the additional approximation property of $E$ and $X$. Indeed, he showed  that:
\begin{theorem}[{Kalton, \cite[Corollary 2]{K}}]\label{Kalton}
If $E$ and $X$ are reflexive  and  $\mathcal B(E,X)=\mathcal K(E,X)$, then $\mathcal{B}(E,X)$  is reflexive.
\end{theorem}
More importantly,  Holub \cite{H} also conjectured that the necessity also remains valid without the approximation property of $E$ and $X$.
 In the following we serveTheorem \ref{ultrafilter} to prove his  conjecture affirmatively and improve the necessity.
\begin{theorem}\label{finite}
Let $E$ and $X$ be  Banach spaces.
 If  $\mathcal{B}(E,X)$  is reflexive, then $E$ and $X$ are reflexive and  $\mathcal{B}(E,X)=\mathcal{A}(E,X).$
\end{theorem}
\begin{proof}
 Let $\mathcal{B}(E,X)$ be reflexive, then the reflexivity of $E$ and $X$  follows from the fact that both $E^*$ and $X$ can be  naturally embedded in $\B(E,X)$.
  To prove the identity   $\mathcal{B}(E,X)=\mathcal{A}(E,X)$, we assume   that $E$ is infinite dimensional.  By Corollary \ref{aultrafilter} there is an isometric isomorphism $\Psi_0: (E\widehat{\otimes} X^*)_\mathcal{U}/W_0\to \mathcal{A}(E,X)^*$, for some ultrafilter $\mathcal{U}.$  On the other hand, $\B(E,X)$ can be easily embedded in $((E\widehat{\otimes} X^*)_\mathcal{U}/W_0)^*$ via the mapping  $\theta :\B(E,X)\to ((E\widehat{\otimes} X^*)_\mathcal{U}/W_0)^*$ with the rule $\theta(T)((u_\alpha)_\mathcal{U}+W_0)=\lim_\mathcal{U}u_\alpha (T)$ for each $T\in \B(E,X)$ and $(u_\alpha)_\mathcal{U}+W_0\in (E\widehat{\otimes} X^*)_\mathcal{U}/W_0$. 
Now if we consider the mapping
\[
\mathcal{B}(E,X)\xhookrightarrow{\theta}\left((E\widehat{\otimes} X^*)_\mathcal{U}/W_0\right)^*\xrightarrow{\Psi_0^{*-1}} \mathcal{A}(E,X)^{**}\cong\mathcal{A}(E,X),
\]
then a direct verification reveals that  $\Psi_0^{*-1}\left( \theta(T)\right)=T,$ for every $T\in \mathcal{B}(E,X). $ Indeed, for every $\phi\in \mathcal{A}(E,X)^*$ with $\Psi_0^{-1} (\phi)=\left(\sum_{i=1}^\infty a_{\alpha,i}\otimes x^*_{\alpha,i}\right)_\mathcal{U}+W_0$, we have
\begin{align*}
\langle \Psi_0^{*-1}\left(\theta(T)\right), \phi\rangle 
=\big\langle \theta(T), \big(\sum_{i=1}^\infty a_{\alpha,i}\otimes x^*_{\alpha,i}\big)_\mathcal{U}+W_0\big\rangle
=\lim_\mathcal{U}\sum_{i=1}^\infty x^*_{\alpha,i}(T(a_{\alpha,i}))
=\langle T, \phi\rangle.
\end{align*}
Therefore, $\mathcal{B}(E,X)=\mathcal{A}(E,X)$, as claimed. 
\end{proof}
As an immediate consequence  of Theorem \ref{finite}, we get  the following result characterizing   the reflexivity of $\B(E)$,  (see \cite [Corollary 2]{Ba}).
\begin{corollary}\label{finite1}
Let $E$ be a Banach space. Then $\mathcal{B}(E)$ is reflexive if and only if $E$ is finite-dimensional.
\end{corollary}
A classical theorem of Pitt \cite{P} (see also \cite[p.175]{FZ})  asserts that $\mathcal{B}(\ell^p,\ell^q)=\mathcal{K}(\ell^p,\ell^q)$, when  $1<q<p<\infty$.  Combining this together with Theorems \ref{Kalton} and \ref{finite}, we arrive at the following improvement of Pitt's theorem.
\begin{corollary}
$\mathcal{B}(\ell^p,\ell^q)=\mathcal{A}(\ell^p,\ell^q)$, when  $1<q<p<\infty$.
\end{corollary}\medskip
\subsection*{Towards some partial converses of Corollary \ref{SD}}
We conclude this section with Theorem \ref{CSD} and Corollary \ref{CSD1} presenting some partial converses of Corollary \ref{SD}. To do this, we need the following auxiliary lemma.
%

\begin{theorem}\label{CSD}
Suppose that $E$ and $X$ are Banach spaces such that the identity $E\widehat{\otimes}X^*\cong \mathcal B(E,X)^*$ holds.
\begin{enumerate}[\hspace{1em}\rm (i)]
\item If $X$ is reflexive, then  $E$ is reflexive.
\item  If $X$ is non-reflexive and $\mathcal B(E,X)=\mathcal A(E,X)$, then $E$ is finite-dimensional.
\end{enumerate}
\end{theorem}
\begin{proof}
(i) Suppose that $X$ is reflexive. Then from the identity $E\widehat{\otimes}X^*\cong \mathcal B(E,X)^*$, we arrive at
\[\mathcal B(E,X)\cong\mathcal B(E,X^{**})\cong(E\widehat{\otimes} X^*)^*\cong\mathcal B(E,X)^{**}, \]
and this shows that $\mathcal B(E,X)$ is reflexive. Now, as $E^*$ can  be embedded in $\B(E,X)$, it follows that $E$ is also reflexive. 

(ii) Suppose on the contrary that $E$ is infinite-dimensional. Then, by Theorem \ref{ultrafilter}, there exists an ultrafilter $\mathcal{U}$  such that $(E\widehat{\otimes} X^*)_\mathcal{U}/W_0\cong\mathcal A(E,X)^*,$ and Remark \ref{RS}(b) confirms that $\mathcal U$ is free and countably incomplete. From the assumption we then have
\[ E\widehat{\otimes}X^*\cong \mathcal B(E,X)^*= \mathcal A(E,X)^*\cong (E\widehat{\otimes} X^*)_\mathcal{U}/W_0.\]
A direct verification based on the above identification now shows that the canonical embedding $E\widehat{\otimes}X^*\hookrightarrow (E\widehat{\otimes} X^*)_\mathcal{U}/W_0$ is onto.

Suppose that  $(v_n)$ is a bounded sequence in  $E\widehat{\otimes}X^*$. Since $\mathcal{U}$ is countably incomplete, there exists a sequence $(I_n)$ of the subsets of $I$ such that $I_{n+1}\subseteq I_{n}$, for all $n$, and $\cap_{n=1}^\infty I_n=\emptyset$. We construct a net $(u_\alpha)$ in $E\widehat{\otimes}X^*$ by $u_\alpha=v_n$ when $i\in I_n\setminus I_{n+1}$ and $u_\alpha=0$ for $i\in I\setminus I_1$.
Since the canonical embedding $E\widehat{\otimes}X^*\hookrightarrow (E\widehat{\otimes} X^*)_\mathcal{U}/W_0$ is onto,  there exists $u\in E\widehat{\otimes}X^*$ such that $(u_\alpha-u)_\mathcal{U}\in W_0$, and it follows that $\lim_\mathcal{U}(u_\alpha-u)(T)=0$ for every  $T\in\mathcal{A}(E,X)=\mathcal{B}(E,X)$. This implies that $\text{weak-}\lim_\mathcal{U} u_\alpha =u$,  so the sequence $(v_n)$ has a weak cluster point. Now, Eberlein-$\check{\text{S}}$mulian theorem  (see \cite[V, 13.1]{C}) implies that $E\widehat{\otimes} X^*$ is reflexive. From the embedding  $\mathcal{B}(E,X)\hookrightarrow\mathcal{B}(E,X)^{**}\cong(E\widehat{\otimes}X^*)^*,$  we get  that $\mathcal{B}(E,X)$  is reflexive which in turn implies the reflexivity of $X$, that  is a contradiction.\smallskip
\end{proof}
It is worth to mention that in the proof of the second part of the latter result we have applied  Theorem \ref{ultrafilter}, which is a strengthened version of Corollary \ref{SD} itself.



\begin{corollary}\label{CSD1}
Suppose that $E$ is a Banach space. If  for every Banach space $X$ the identity $E\widehat{\otimes}X^*\cong \mathcal B(E,X)^*$ holds, then $E$ is finite-dimensional.
\end{corollary}
\begin{proof}
If we use the  identity $E\widehat{\otimes}X^*\cong \mathcal B(E,X)^*$  for a reflexive Banach space $X$, then Theorem \ref{CSD}(i) implies that $E$ is  reflexive. If we serve the latter identity with $X=E$, we get
\[
\mathcal{B}(E)^{**}\cong (E\widehat{\otimes}E^*)^*\cong \mathcal{B}(E,E^{**})\cong \mathcal{B}(E).
\]
Thus, Corollary \ref{finite1} implies that $E$ is finite dimensional.
\end{proof}
\section{Some more applications of the  PLR}
We close the paper with some more  by-products of  the PLR, Theorem \ref{PLR}, that are motivated  by the techniques given in \cite{D}. We indeed  present some generalizations of the well-known  Goldstien weak$^*$-density theorem stating that, for a normed space $X,$ Ball$(X)$ is weak$^*$-dense in Ball$(X^{**})$ (see \cite[V, Proposition 4.1]{C}).

Before proceeding, let us recall the operator weak$^*$ ($ow^*$, for short)-topology on $\mathcal B(E,X^{**})$. For two normed spaces $E$ and $X,$ the $ow^*$-topology on $\mathcal B(E,X^{**})$ is induced by the family $\{e\otimes f\}_{e\in E, f\in X^*}$ of seminorms, where $(e\otimes f)(T)=|(T(e))(f)|$\quad ($T\in \mathcal B(E,X^{**})).$ Consequently, $T_\alpha\xrightarrow{ow^*} T$ if and only if $T_\alpha (e) \xrightarrow{{\rm weak}^*}T(e)$ for every $e\in E$.

We also need to recall the notion of injectivity for a normed space. A normed space $X$ is said to be injective if for every pair of normed spaces $F, E$ with $F\subseteq E$, every operator in $\B(F, X)$ has an extension in $\B(E, X)$ with the same norm.

 We come with the following auxiliary lemma.
\begin{lemma}\label{ext}
Let $E$ and $X$ be two Banach spaces and let $F$ be a finite-dimensional subspace of $E$. Then every bounded linear operator $T:F\longrightarrow X$ can be extended to a bounded linear operator $\overline{T}:E\to X.$
\end{lemma}
\begin{proof}
There exists a closed subspace  $F'$ of $E$ such that $E=E\oplus F'.$ We claim that the set $\{x\in F: x+x'\in \hbox{Ball}(E), \hbox{for\ some}\ x'\in F'\}$ is bounded in $E$. Suppose, on the contrary, that  $(x_n+x'_n)$ be a sequence in ${\rm Ball}(E)$ with $(x_n)\subseteq F$ and $(x'_n)\subseteq F'$ such that   $\| x_n\|\to\infty$. Since $F$ is finite-dimensional, the sequence $(x_n/\|x_n\|)$ has a convergent subsequence. We, may,  assume that $x_n/\|x_n\|\to x_0$, then $\|x_0\|=1$. On the other hand  $(x_n+x'_n)/\|x_n\|\to 0$, so  we arrive at $x'_n/\|x_n\|\to -x_0$. This follows that $x_0\in F\cap F'=\{ 0\}$, which is a contradiction. Suppose that $M>0$ is the required  bound for the above mentioned subset of $F$.

 Now, if $T:F\to X$ is a bounded linear operator, then $T$ can be extended to a bounded linear operator $\overline{T}:E\to X$. Indeed, we just need to define $\overline{T}(x+x')=T(x)$, for each $x\in F$ and $x'\in F'$. Thus
 \[
 \|\overline{T}\| =\sup_{\|x+x'\|\leq 1} \|\overline{T}(x+x')\|= \sup_{\|x+x'\|\leq 1} \|T(x)\|\leq \sup_{\|x\|\leq M}\| T(x)\|\leq M\|T\|.
 \]
 Thus $\overline T$ is bounded, as required.
\end{proof}
We are now ready to give the following result, the  part (iii)  of which is an improvement of a result of  Lima and Oja {\cite[Corollary 2.11]{LO}}.
\begin{theorem} \label{Goldstine}
Let $E$ and $X$ be two normed spaces. The following statements hold.
\begin{enumerate}[\hspace{1em}\rm (i)]
\item $\overline{\B(E,X)}^{ow^*}=\B(E,X^{**}).$
\item $\overline{{\rm Ball}(\B(E,X))}^{ow^*}={\rm Ball}(\B(E,X^{**})),$ when $X$ is injective.
\item $\overline{{\rm Ball}(\F(E,X))}^{ow^*}={\rm Ball}(\F(E,X^{**}));$.
\end{enumerate}
\end{theorem}
\begin{proof}
Let $T\in \B(E,X^{**})$ with $\|T\|=1$ and let $I$ be the set of all
\[
N(T,A,B,\varepsilon)=\lbrace S : |(S-T) (a)(f)|<\varepsilon, a\in A, \ f\in B\rbrace,
\]
where $0<\varepsilon\leq 1$ and $A, B$ are finite subsets of ${\rm Ball(E)}$ and ${\rm Ball(X^*)},$ respectively.
We define the  partial order $\leq$ on $I$ by
\[
N(T,A_1,B_2,\varepsilon_1)\leq N(T,A_2,B_2,\varepsilon_2) \quad \text{if and only if} \quad A_1\subseteq A_2, B_1\subseteq B_2, \varepsilon_2\leq\varepsilon_1.
\]
Then for every $N(T, A_\alpha,B_\alpha,\varepsilon_\alpha)\in I,$ by the PLR, (Theorem \ref{PLR}), there exists a one-to-one linear map $S_\alpha:T(\langle A_\alpha\rangle)\to X$ such that $\|S_\alpha\|, \|S_\alpha^{-1}\|<1+\varepsilon,$ $S_\alpha|_{T(\langle A_\alpha\rangle)\cap X}=id_{T(\langle A_\alpha\rangle)\cap X}$, and $f(S_\alpha(e))=e(f)$ for all $e\in T(\langle A_\alpha\rangle)$ and $f\in B.$\smallskip

To prove (i),  for each $\alpha$ let $\overline{S_\alpha}: X^{**}\to X$ be that  extension of $S_\alpha$ whose existence was  guaranteed  by Lemma  \ref{ext}. Then we  define the  bounded linear operator $T_\alpha:=(1+\varepsilon_\alpha)^{-1}\overline{S_\alpha}\circ T$.  The construction of $(T_\alpha)$ implies that $T_\alpha\in N(T,A_\beta,B_\beta,\varepsilon_\beta)$ when $\alpha\geq\beta,$ from which we have $T_\alpha\xrightarrow{ow^*} T$.\smallskip

For (ii), since $X$ is injective, the operator  $S_\alpha$ can be extended to a bounded linear operator $\overline{S_\alpha}: X^{**}\to X$ with the same norm,  for each $\alpha\in I$. If we set $T_\alpha :=(1+\varepsilon_\alpha)^{-1}\overline{S_\alpha}\circ T$, then it  is clearly a  bounded linear operator with $\|T_\alpha\|\leq \|T\|$. We further have the inequality
  \[\|T_\alpha\|=\|(1+\varepsilon_\alpha)^{-1}\overline{S_\alpha}\circ T\|\geq (1+\varepsilon_\alpha)^{-2} \|T\|, \] from which we get $\|T_\alpha\|\to\|T\|$. The construction of $(T_\alpha)$ now  implies that $T_\alpha\in N(T,A_\beta,B_\beta,\varepsilon_\beta)$ when $\alpha\geq\beta.$ We thus have $T_\alpha\xrightarrow{ow^*} T$.\smallskip

For (iii), we just need to note that when $T\in \F(E,X^{**})$, then there exists a finite subset $A_0$ of $E$ such that  $T(\langle A_0\rangle)=T(E)$. Fix some $f\in {\rm Ball(X^*)},$ and set   $T_\alpha:=(1+\varepsilon_\alpha)^{-1}S_\alpha\circ T,$ for each  $\alpha$ with $N(T, A_\alpha,B_\alpha,\varepsilon_\alpha)\geq N(T, A_0,\{ f\},1)$. Then it follows from the definition that   $T_\alpha\in \F(E,X)$, and a similar argument as part (ii) reveals that $\|T_\alpha\|\to\|T\|$ and $T_\alpha\xrightarrow{ow^*} T$, as claimed.

\end{proof}


\end{document}